\documentclass{article}

\usepackage{authblk}

\usepackage{amsmath}
\usepackage{amsthm}
\usepackage{amsfonts}
\usepackage{amssymb}
\usepackage[all]{xy}
\usepackage{url}

\DeclareMathSymbol{\rightrightarrows}  {\mathrel}{AMSa}{"13}

\def\Re{\operatorname{Re}}

\catcode`\@=11
\def\varholim@#1#2{\mathop{\vtop{\ialign{##\crcr
        \hfil$#1\m@th\operator@font holim$\hfil\crcr
 \noalign{\nointerlineskip\kern\ex@}#2#1\crcr
 \noalign{\nointerlineskip\kern-\ex@}\crcr}}}}
\def\hocolim{\mathpalette\varholim@\rightarrowfill@} 
\def\hoinvlim{\mathpalette\varholim@\leftarrowfill@}
\catcode`\@=\active

\newtheorem{theorem}{Theorem}
\newtheorem{lemma}[theorem]{Lemma}

\newtheorem{proposition}[theorem]{Proposition}

\theoremstyle{definition}

\newtheorem{example}[theorem]{Example}

\newtheorem{remark}[theorem]{Remark}

\begin{document}

\title{\bf Stability for UMAP}
\author{J.F. Jardine\thanks{Supported by NSERC.}}

\affil{\small Department of Mathematics\\University of Western Ontario\\
  London, Ontario, Canada
}
\affil{jardine@uwo.ca}

\maketitle

\begin{abstract}
  \noindent
This paper displays the Healy-McInnes UMAP construction V(X,N) as an iterated ushout of Vietoirs-Rips objects $V(X,D_{x})$, which are associated to extended pseudo metric spaces (ep-metric spaces) defined by a system $N$ of neighbourhoods of the elements of a finite set $X$. An inclusion of finite sets $X \subset Y$ defines a map of UMAP systems $V(X,N) \to V(Y,N')$ in the presence of a compatible system of neighbourhoods $N'$ for $Y$. There is also an induced map of ep-metric spaces $(X,D) \to (Y,D')$, where $D$ and $D'$ are colimits (global averages) of the metrics defined by the respective neighbourhood systems. We prove a stablity result for the restriction of this ep-metric space map to global components. This stability result translates, via excision for path components, to a stability result for global components of the UMAP systems.
\end{abstract}

\nocite{fuzzy-Spivak}

\bigskip 

The main result of \cite{metric} says that if $X$ is a finite extended pseudo-metric space (ep-metric space), then the canonical map
\begin{equation*}
  \eta: V(X)_{s} \to S(X)_{s}
  \end{equation*}
is a weak equivalence for all distance parameters $s$. Here, $V(X)$ is the Vietoris-Rips system and $X \mapsto S(X)$ is the singular functor.

In this paper, we use this result to model the UMAP construction, and we prove a stability result for the resulting hierarchies of clusters.
\medskip

For the general program, we start with sets $N_{x}$ (disjoint from $x$) for each $x \in X$, and distances (or weights) $d_{x}(x,y) \geq 0$ for all $y \in N_{x}$. These distances canonically extend to an ep-metric space structure $(U_{x},D_{x})$ on the set
\begin{equation*}
  U_{x} = \{x\} \sqcup N_{x},
\end{equation*}
  and then to an ep-metric space structure $(X,D_{x})$ on all of $X$, for which $D_{x}(y,z) = \infty$ unless both $y$ and $z$ are in $U_{x}$. The metric space structures $(X,D_{x})$ can be glued together along ep-metric space morphisms $(X,\infty) \to (X,D_{x})$ to produce an ep-metric space
\begin{equation*}
  (X,D) = \vee_{x \in X}\ (X,D_{x}).
\end{equation*}
  Similarly, the Vietoris-Rips systems $V(X,D_{x})$ can be glued together along the maps $X \to V(X,D_{x})$ to produce a system
\begin{equation*}
  V(X,N) = \vee_{x \in X}\ V(X,D_{x}).
\end{equation*}

The notation $V(X,N)$ reflects the fact that this system of spaces depends on the family $N = \{N_{x},\ x \in X\}$ of neighbourhoods, which includes choices of weights $d_{x}$ within each neighbourhood $N_{x}$.

The object $V(X,N)$, for suitable choices of neighbourhoods and weights, gives the various models for the UMAP system.

The original UMAP system $S(X,N)$ of Healy and McInnes \cite{HMc-2018}, is constructed from Spivak's singular functor \cite{fuzzy-Spivak}, \cite{metric}, with
\begin{equation*}
  S(X,N) := \vee_{x \in X}\ S(X,D_{x}).
\end{equation*}
There is a sectionwise weak equivalence $V(X,N) \to S(X,N)$ by the main result of \cite{metric}, and we use the Vietoris-Rips construction $V(X,N)$ since it is more familiar and easier to manipulate.

The choice of neighbourhood sets $N_{x}$ can be arbitrary, but in \cite{HMc-2018} it is the set of $k$-nearest neighbours. The collections of distances $d_{x}(x,y)$ are also arbitrary, but are defined in \cite{HMc-2018}, variously, as the original distance $d_{x}(x,y) = d(x,y)$ or the probability $d_{x}(x,y) = \frac{1}{r_{x}}d(x,y)$, or $d_{x}(x,y) = \frac{1}{r_{x}}(d(x,y) - s_{x})$. Here, $r_{x} = \max_{y \in N_{x}} d(x,y)$ and $s_{x} = \min_{y \in N_{x}} d(x,y)$.

All corresponding constructions $V(X,N)$ are variants of the UMAP construction, and they are easily compared. The shrpest results on the general structure of $V(X,N)$ require the weights $d_{x}(x,y) > 0$ for $y \in N_{x}$, and this is assumed for most of the paper.
\medskip

Suppose given an ep-metric space map $i: (X,d_{X}) \to (Y,d_{Y})$, where the underlying function is an injection, and $X$ and $Y$ are finite. The assumption that $i$ is an ep-metric space morphism means that $i$ compresses distance in the sense that $d_{Y}(i(x),i(y)) \leq d_{X}(x,y)$ for all $x,y \in X$.

We can assume for now that $X$ and $Y$ are metric spaces, and are therefore {\it globally connected} in the sense that $d(x,y) < \infty$ for all $x,y$. In that case, since $X$ is finite, we define the {\it compression factor} $m(i)$ by
\begin{equation*}
m(i) = \max_{x \ne y}\ \frac{d_{X}(x,y)}{d_{Y}(i(x),i(y))}
\end{equation*}
This makes sense because none of the distances in the ratio are either $0$ or $\infty$.

If we further assume that for every $y \in Y$ there is an $x \in X$ such that $d_{Y}(y,x) \leq r$, then the same argument as for the ordinary Rips stability theorem produces a homotopy interleaving
\begin{equation*}
  \xymatrix{
    V(X)_{s} \ar[r]^-{\sigma} \ar[d]_{i} & V(X)_{m(i)(s+2r)} \ar[d]^{i} \\
    V(Y)_{s} \ar[r]_-{\sigma} \ar[ur]^{\theta} & V(Y)_{m(i)(s+2r)}
  }
  \end{equation*}
This statement appears as Proposition \ref{prop 8} in this paper.

Suppose now that $i: X \subset Y$ is an inclusion of finite sets, and we have made choice of neighbourhoods $N_{x},\ x \in X$ and $N_{y},\ y \in Y$. Suppose that \begin{itemize}
\item[1)] the inclusion $i$ induces inclusions $i: N_{x} \subset N_{i(x)}$, and
\item[2)] the weights are chosen such that $d_{x}(x,x') > 0$ for $x' \ne x \in N_{x}$, $d_{y}(y,y') > 0$ for $y' \ne y \in N_{y}$, and $d_{i(x)}(i(x),i(y)) \leq d_{x}(x,y)$ for all $y \in N_{x}$.
  \end{itemize}
The assumptions imply that the inclusion $i$ induces an ep-metric space map $i: (X,D) \to (X,D')$, and the global connected components of both ep-metric spaces are metric spaces. If $E$ is a global connected component of $(X,D)$, then there is a global connected component $F$ of $(Y,D')$ such that $i$ restricts of a ep-metric space morphism $i: (E,D) \to (F,D')$ of metric spaces.

Subject to the assumptions of the last paragraph, it follows from Proposition \ref{prop 8} that, if for every $y \in F$ there is an $x \in E$ such that $d_{Y}(y,i(x)) \leq r$, then there is a homotopy interleaving
\begin{equation}\label{eq 1}
  \xymatrix{
    V(E,D)_{s} \ar[r]^-{\sigma} \ar[d]_{i} & V(E,D)_{m(i)(s+2r)} \ar[d]^{i} \\
    V(F,D')_{s} \ar[r]_-{\sigma} \ar[ur]^{\theta} & V(F,D')_{m(i)(s+2r)}
  }
  \end{equation}

This is a componentwise stability result for the ep-metric space morphism $i: (X,D) \to (Y,D')$, which appears as Theorem \ref{th 9} in this paper. The input for this result involves compatible choices of neighbourhoods and weights within those neighbourhoods, rather than distance.

The canonical ep-metric space maps $(X,D_{x}) \to (X,D)$ induce a map of systems
\begin{equation*}
  \phi: V(X,N) = \vee_{x \in X}\ V(X,D_{x}) \to V(X,D)
\end{equation*}
which is natural with respect to inclusions $i: X \subset Y$ satisfying the conditions above. 

The ep-metric space $(X,D)$ is a disjoint union of its global connected components $E$, and the system $V(X,D)$ is a disjoint union of the systems $V(E,D)$. This splitting defines a disjoint union structure
\begin{equation*}
  V(X,N) = \bigsqcup_{E}\ V(X,N)(E),
\end{equation*}
where $V(X,N)_{E}$ is the pullback of the system $V(E)$ under the map $\phi$.
The induced map
\begin{equation*}
  \phi_{\ast}: \pi_{0}V(X,N)(E) \to \pi_{0}V(E)
\end{equation*}
is an isomorphism of systems of sets, by path component excision (Lemma \ref{lem 2}).

The componentwise stability result displayed in the interleaving (\ref{eq 1}) therefore specializes to interleavings in clusters
\begin{equation}\label{eq 2}
  \xymatrix{
    \pi_{0}V(X,N)(E)_{s} \ar[r]^-{\sigma} \ar[d]_{i} & \pi_{0}V(X,N)(E)_{m(i)(s+2r)} \ar[d]^{i} \\
    \pi_{0}V(Y,N')(F)_{s} \ar[r]_-{\sigma} \ar[ur]^{\theta} & \pi_{0}V(Y,N')(F)_{m(i)(s+2r)}
  }
  \end{equation}
This is a stability result for UMAP, which appears as Theorem \ref{th 10} below. Theorem \ref{th 10} is the main result of this paper.

 \tableofcontents

\section{General constructions}

    Suppose that we have a set $X$ with a finite list of ep-metric space structures $(X,d_{i})$, $i=1, \dots ,k$. We can also endow $X$ with a discrete ep-metric space structure, so that $d_{\infty}(x,y) = \infty$ for all $x,y \in X$. Suppose that $X$ has a total ordering.
\medskip

There are canonical ep-metric space morphisms $(X,d_{\infty}) \to (X,d_{i})$, all of which are the identity on $X$. Write $(X,D)$ for the colimit in $ep-\mathbf{Met}$, giving a diagram 
\begin{equation*}
  \xymatrix{
    (X,d_{\infty}) \ar[r] \ar[d] & (X,d_{i}) \ar@{.>}[d]^{\tau_{i}} \\
    (X,d_{j}) \ar@{.>}[r]_{\tau_{j}} & (X,D)
  }
  \end{equation*}
The maps $\tau_{i}: (X,d_{i}) \to (X,D)$ are the canonical maps into the colimit.  Recall \cite{metric} that the colimit $(X,D)$ is formed by taking the colimit of the underlying functions, and endowing it with a metric, in this case $D$. The colimit of functions, which are identity functions on $X$, is $X$ again, so that the notation $(X,D)$ makes sense.

We also write
\begin{equation*}
  (X,D) = \vee_{i}\ (X,d_{i})
\end{equation*}
to reflect the fact that we are gluing together the ep-metric spaces $(X,d_{i})$ along the underlying set $X$.

Formally,
\begin{equation*}
  D(x,y) = \inf_{P}\ (\sum\ d_{i_{j}}(x_{j},x_{j+1})),
\end{equation*}
indexed over all polygonal paths
\begin{equation*}
  x=x_{0},x_{1}, \dots ,x_{n}=y
\end{equation*}
and choices of metrics $d_{i_{j}}$ in the list $d_{i},\ 1 \leq i \leq k$. The pair $x,y$ forms a polygonal path, so that
\begin{equation*}
  D(x,y) \leq d_{i}(x,y)
\end{equation*}
for all $i$. In this sense, the ep-metric $D$ optimizes the metrics $d_{i}$.

There may not be a polygonal path $P$ and metrics $d_{i_{j}}$ such that all $d_{i_{j}}(x_{i},x_{i+1})$ are finite. In that case, we have $D(x,y) = \infty$.
\medskip

If $X$ is a finite set, then the collection of polygonal paths from $x$ to $y$ in $X$ is finite, and so
\begin{equation*}
  D(x,y) = \sum\ d_{i_{j}}(x_{j},x_{j+1})
\end{equation*}
for some choice of polygonal path $P$ and metrics $d_{i_{j}}$. In that case, $d_{i_{j}}(x_{j},x_{j+1})$ must be minimal among all $d_{k}(x_{j},x_{j+1})$.
\medskip

The maps $(X,d_{\infty}) \to (X,d_{i})$ induce maps $X \to V(X,d_{i})$ into Vietoris-Rips systems, and we form the iterated pushout
\begin{equation}\label{eq 3}
 V(X,D) := \vee_{i}\ V(X,d_{i})
\end{equation}
in the category of systems. This means that the object (\ref{eq 3}) is the colimit of all maps
\begin{equation}\label{eq 4}
  X \to V(X,d_{i}),
\end{equation}
over the discrete system $X$. The maps (\ref{eq 4}) are sectionwise monomorphisms, so the object $V(X,D)$ is a type of homotopy colimit.
\medskip

\begin{remark}
  In practice and in general, although one tends to be notationally lazy, it is better to replace the Vietoris-Rips system $s \mapsto V_{s}(X)$ with the homotopy equivalent system $s \mapsto BP_{s}(X)$, where $P_{s}(X)$ is the poset of non-degenerate simplices of $V_{s}(X)$, and $BP_{s}(X)$ is the nerve of $P_{s}(X)$. The poset $P_{s}(X)$ can be described explicitly as the collection of subsets $\sigma$ of $X$ such that $d(x,y) \leq s$ for all $x,y \in \sigma$. The structure of the poset $P_{s}(X)$ does not depend on an ordering of the set $X$.

The ep-metric space maps $(X,d_{i}) \to (X,D)$ induce commutative diagrams
\begin{equation*}
  \xymatrix{
    BP(X,d_{i}) \ar[r] \ar[d]_{\gamma}^{\simeq} & BP(X,D) \ar[d]^{\gamma_{\ast}} \\
    V(X,d_{i}) \ar[r] & V(X,D)
  }
  \end{equation*}
of maps of systems, where the map $\gamma$ is a sectionwise weak equivalence defined by subdivision, and the induced map $\gamma_{\ast}$ is a sectionwise weak equivalence arising from the displayed comparison of homotopy colimits.

From this perspective, we can write 
\begin{equation*}
  V(X,D) = BP(X,D) = \vee_{i}\ BP(X,d_{i}) = \vee_{i}\ V(X,d_{i})
\end{equation*}
as sectionwise homotopy types.
\end{remark}

\begin{lemma}[Excision]\label{lem 2}
  Suppose that $X$ is a finite set, with a finite collection of ep-metric structures $d_{i}$.
  
  Then the canononical map
  \begin{equation*}
   \phi: \vee_{i}\ V(X,d_{i}) \to V(X,D)
  \end{equation*}
  induces bijections
\begin{equation*}
  \phi_{\ast}:  \pi_{0}(\vee_{i}\ V(X,d_{i}))_{s} \xrightarrow{\cong}
  \pi_{0}V(X,D)_{s}
  \end{equation*}
  for all $s$.
\end{lemma}

\begin{proof}
  The map $\phi$ is the identity on vertices, so that $\phi_{\ast}$ is surjective.

  Suppose that $D(x,y) \leq s$ in $(X,D)$. There is a polygonal path
  \begin{equation*}
    P: x=x_{0},x_{1}, \dots ,x_{n}=y
  \end{equation*}
  and metrics $d_{i_{j}}$ such that
  \begin{equation*}
    D(x,y) = \sum_{j}\ d_{i_{j}}(x_{j},x_{j+1}) \leq s,
  \end{equation*}
  since $X$ is finite.
  This means that $d_{i_{j}}(x_{j},x_{j+1}) \leq s$ for all $j$, and so there are $1$-simplices $(x_{j},x_{j+1})$ in $V(X,d_{i_{j}})_{s}$ which together describe a path from $x$ to $y$ in $\vee_{X}\ V_{s}(X,d_{i})$.

  It follows that, if $x,y$ are in the same path component of $V(X,D)_{s}$, then $x,y$ are in the same path component of $\vee_{i} V(X,d_{i})_{s}$. 
  \end{proof}

\section{UMAP}

    The UMAP algorithm of \cite{HMc-2018} starts with a finite metric space $X$. We assume that $X$ has a total ordering.

For each point $x \in X$ one finds the list
\begin{equation*}
  N_{x} := \{ x_{1}, \dots ,x_{k}\}
\end{equation*}
of distinct $k$-nearest neighbours with $x_{i} \ne x$, with maximum distance $r_{x} = \max_{i}\ d(x,x_{i})$.

The set $N_{x}$ is the set of {\it neighbours} of $x$.
\medskip

In much of what follows, the choices of the sets $N_{x}$ can be quite arbitrary.
 In all applications, one assigns distances $d_{x}(x,y)$ for all neighbours $y \in N_{x}$, and then one extends functorially to an ep-metric $D_{x}$ on $X$. This is done for all $x \in X$.
 \medskip

 \noindent
 {\bf Examples}:\
    Possibilities for $d_{x}(x,y)$ include $\frac{1}{r_{x}}d(x,y)$, $\frac{1}{r_{x}}(d(x,) - \eta_{x})$ where $\eta_{x}$ is the distance from $x$ to a nearest neighbour. We can also use the ambient metric $d_{x}(x,y)=d(x,y)$ from $X$.
    \medskip

\begin{remark}\label{rem 3}    
Explicitly, given $x \in X$ we find a set (and a listing) $N_{x}=\{ x_{1}, \dots ,x_{k}\}$ of $k$-nearest neighbours, by finding an element $x_{1}$ (in the total order) such that $d(x,x_{1})$ is minimal ($x_{1}$ is a nearest neighbour).  Then $x_{2} \in X - \{x,x_{1}\}$ is chosen such that $d(x,x_{2})$ is minimal and $x_{2}$ is the first element in the total order that has this property, and so on.

The algorithm is set up such that the sublist $\{x_{i},x_{i+1}, \dots ,x_{k}\}$ of elements having $d(x,x_{j}) = r$ has $x_{i} < x_{i+1} < \dots < x_{k}$ in the total order.
\end{remark}

\noindent
{\bf Assumptions}:\
Suppose that $X$ is a finite set. Suppose given a system of neighbourhoods $N_{x}$ for $x \in X$, and define distances $d_{x}(x,y) > 0$ for each $y \in N_{x}$.
\medskip

One defines an ep-metric $D_{x}$ first on the set
\begin{equation*}
  U_{x}= \{x\} \sqcup N_{x}
\end{equation*}
and then one extends to all of $X$ with the decomposition
\begin{equation}\label{eq 5}
  X = U_{x} \sqcup (\bigsqcup_{y \in X-U_{x}}\ \{y\}).
\end{equation}

The ep-metric space structure on the set $U_{x}$ is given by the wedge
\begin{equation*}
  (U_{x},D_{x}) = \vee_{y \in N_{x}}\ (\{ x,y \},d_{x})
\end{equation*}
over $x$ of the $2$-element metric spaces $(\{x,y\},d_{x})$, in the category of ep-metric spaces. The metric $D_{x}$ on $U_{x}$ has the property that $D_{x}(x,y) = d_{x}(x,y)$ for $y \in N_{x}$. The triangle inequality forces
\begin{equation*}
  D_{x}(y,z) \leq d_{x}(y,x) + d_{x}(x,z)
\end{equation*}
for $y \ne z$ in $N_{x}$.
At the same time, the sum $d_{x}(y,x) + d_{x}(x,z)$ is the length of the shortest polygonal path $(y,x,z)$ between $y$ and $z$ in $U_{x}$, and so it follows that
\begin{equation*}
  D_{x}(y,z) = d_{x}(y,x) + d_{x}(x,z)
\end{equation*}
for $y \ne z$ in $N_{x}$.

Use the decomposition (\ref{eq 5}) to extend $D_{x}$ to an ep-metric on all of $X$. This forces $D_{x}(y,z) = \infty$ unless $y$ and $z$ are both in $U_{x}$. 
\medskip

Define systems of simplicial sets $V(X,N)$ and $S(X,N)$ by setting
\begin{equation*}
  V(X,N) = \vee_{x \in X} V(X,D_{x})
\end{equation*}
and
\begin{equation*}
 S(X,N) = \vee_{x \in X} S(X,D_{x}),
\end{equation*}
respectively. Here, $V(X,N)$ is the iterated pushout of the cofibrations $X \to V(X,D_{x})$, where the set $X$ is identified with a constant, discrete system. Similarly, $S(X,N)$ is the iterated pushout of the cofibrations $X \to S(X,D_{x})$.

The maps $\eta: V(X,D_{x}) \to S(X,D_{x})$ are sectionwise weak equivalences by \cite{metric}, and therefore induce a sectionwise weak equivalence
\begin{equation}
  \eta: V(X,N) \to S(X,N)
\end{equation}
by comparison of iterated pushouts (or homotopy colimits).
\medskip

Spivak's realization construction $\Re$ preserves colimits,  and there is a natural isomorphism $\Re(V(X,D_{x})) \cong (X,D_{x})$ (see \cite{metric}), so that the realization
\begin{equation*}
  \Re(V(X,D)) \cong \vee_{X}\ (X,D_{x}) = (X,D)
  \end{equation*}
is the iterated pushout of the maps $X \to (X,D_{x})$ in the ep-metric space category, as in the first section. 
\medskip

\begin{remark}\label{rem 4}
  The set $X$ is finite. The distances $d_{x}$ have the property that $d_{x}(x,y) > 0$ for all $y \in N_{x}$, $x \in X$, and one can show that $D(u,v) = 0$ in $(X,D)$ forces $u=v$.

  In effect, $D(u,v)$ is a sum
  \begin{equation*}
    D(u,v) = \sum D_{z_{i}}(x_{i},x_{i+1})
  \end{equation*}
  which is defined by a particular polygonal path $P: u=x_{0}, \dots ,x_{n}=v$ (since there are only finitely many such paths).
  Then $D(u,v) = 0$ forces all
\begin{equation*}
  D_{z_{i}}(x_{i},x_{i+1}) = d_{z_{i}}(x_{i},z_{i}) + d_{z_{i}}(z_{i},x_{i+1})
\end{equation*}
to be $0$, so that $x_{i} = z_{i} = x_{i+1}$ for all $i$, and $u=v$.
  \end{remark}

\begin{remark}
  It is time for a homotopy theory interlude.
  
Suppose that each map
\begin{equation*}
  V = \{0,1,\dots ,n\} \subset \Delta^{n} = X_{i},\ i \geq 0,
\end{equation*}
is the inclusion of the set of vertices $V$ of the standard $n$-simplex $\Delta^{n}$, and let $Y_{k} = X_{0} \cup \dots \cup X_{k}$ be an iterated pushout of $n$-simplices over the common vertex set $V$.

There is a pushout diagram
\begin{equation*}
  \xymatrix{
    V \ar[r] \ar[d] & X_{1} \ar[d] \\
    X_{0} \ar[r] & X_{0} \cup_{V} X_{1}
  }
  \end{equation*}
in which both $X_{0}$ and $X_{1}$ are contractible. It follows that $X_{0} \cup_{V} X_{1}$ has the homotopy type of the suspension $X_{1}/V \simeq \Sigma V$ for a suitable choice of base point of the discrete set $V$ --- choose $0$. Then $V = \{0,1\} \vee \{0, 2\} \vee \dots \vee \{0,n\}$
is a wedge of $n$ copies of $S^{0}$, and $\Sigma V$ is a wedge of $n$ copies of $\Sigma S^{0} = S^{1}$. Thus,
\begin{equation*}
  Y_{1} = X_{0} \cup_{V} X_{1} \simeq S^{1} \vee \dots \vee S^{1}\enskip \text{($n$ summands)}.
\end{equation*}

More generally, consider the space $Y_{k} = X_{0} \cup_{V} X_{1} \cup_{V} \dots \cup_{V} X_{k}$. Collapsing the contractible space $X_{0}$ to a point gives
\begin{equation*}
  Y_{k} \simeq (X_{1}/V) \vee (X_{2}/V) \vee \dots \vee (X_{k}/V).
\end{equation*}
Each $X_{i}/V$ is an $n$-fold wedge of circles, by the above, so that $Y_{k}$ is a $k\cdot n$-fold wedge of circles. 
\medskip

Suppose that the finite set $X$ has $M+1$ elements. Then $V(X,D)_{\infty}$ is an iterated pushout of the maps $X \subset V(X,D_{x})_{\infty}$. Each $V(X,D_{x})_{\infty}$ is a copy of the $M$-simplex $\Delta^{M}$, and each map $X \subset V(X,D_{x})_{\infty}$ is a copy of the inclusion of vertices $\mathbf{M} \subset \Delta^{M}$.

It follows that $V(X,D)_{\infty}$ is a large wedge of circles. Explicitly, there is a weak equivalence
\begin{equation*}
  V(X,D)_{\infty} \simeq \vee_{M^{2}}\ S^{1}.
\end{equation*}
This space is path connected.

The system of path component sets
\begin{equation*}
  s \mapsto \pi_{0}V(X,d)_{s}
\end{equation*}
therefore describes a hierarchy, as in the standard algorithms of topological data analyis. 
\end{remark}

\section{Stability}

Suppose that $(X,d)$ is an ep-metric space, and that $x \in X$. The {\it global connected component} of $x$ is the collection of $y \in X$ such that $d(x,y) < \infty$. Say that $X$ is {\it globally connected} if $d(x,y) < \infty$ for all $x,y \in X$.
\medskip

  Global connectedness has the following general properties:
  \begin{itemize}
    \item[1)]
  Every ep-metric space $(X,d)$ is a disjoint union of its set $\pi_{\infty}(X,d)$ of global components. 
\item[2)]
  An ep-metric space morphism $f:(X,d_{X}) \to (Y,d_{Y})$ preserves global connected components: if $d_{X}(x,y) < \infty$ then
\begin{equation*}
  d_{Y}(f(x),f(y)) \leq d_{X}(x,y) < \infty,
\end{equation*}
  so that $f(y)$ is in the connected component of $f(x)$. We therefore have an induced function $f_{\ast}: \pi_{\infty}(X,d_{X}) \to \pi_{\infty}(Y,d_{Y})$.
\item[3)]
  Every metric space is globally connected.
  \end{itemize}

  \begin{example}\label{ex 6}
     Suppose that $X$ is a finite set with a system of neighbourhoods $N_{x}$ and associated distances $d_{x}$ for all $x \in X$ as in the list of Assumptions above, with the resulting ep-metric space $(X,D)$.

      The ep-metric space $(X,D)$ has the property that $D(x,y) = 0$ forces $x=y$, by Remark \ref{rem 4}. It follows that the global connected components of the ep-metric space $(X,D)$ are metric spaces.
 
Say that a pair of elements $(x,y)$ of $X$ is a {\it neighbourhood pair} if $x \in N_{y}$ or $y \in N_{x}$. The argument of Remark \ref{rem 4} shows that elements $u$ and $v$ of $(X,D)$ are in the same global connected component if and only if there is a polygonal path
\begin{equation*}
  P: u=x_{0},x_{1}, \dots ,x_{n}=v
\end{equation*}
such that each pair $(x_{i},x_{i+1})$ is a neighbourhood pair.
\end{example}

  Suppose that $(X,d_{X})$ and $(Y,d_{Y})$ are finite metric spaces, and that there is a monomorphism $i: X \subset Y$ that defines a map of ep-metric spaces, so that $d_{Y}(x,y) \leq d_{X}(x,y)$ for all $x,y \in X$. Set
  \begin{equation*}
    m(i) = \max_{x \ne y \in X}\ \{\frac{d_{X}(x,y)}{d_{Y}(i(x),i(y))}\}.
    \end{equation*}
  Then $1 \leq m(i) < \infty$ since $X$ is finite.

  The number $m(i)$ is the {\it compression factor} for the monomorphism $i$.

  \begin{example}\label{ex 7}
      Suppose that $i: X \subset Y$ is an inclusion of finite sets, and choose systems of neighbourhoods $N_{x}$, $x \in X$ and $N'_{y}$, $y \in Y$, with distances $d_{x}$ and $d'_{y}$. Suppose that $N_{x} \subset N'_{i(x)}$ for all $x \in X$, and that
        \begin{equation}\label{eq 7}
          d_{i(x)}'(i(x),i(z)) \leq d_{x}(x,z)
\end{equation}
          for all $z \in N_{x}$, and for all $x \in X$.

          Then the inclusions $i: N_{x} \subset N'_{i(x)}$ and the relations (\ref{eq 7}) define system morphisms $V(X,D_{x}) \to V(Y,D'_{i(x)})$ and $V(X,D) \to V(Y,D')$, as well as ep-metric space morphisms $(X,D) \to (Y,D')$.

          The ep-metric space map $(X,D) \to (Y,D')$ is the realization of the system morphism $V(X,D) \to V(Y,D')$.

        The ep-metric space morphism $(X,D) \to (Y,D')$ preserves global connected components, and the global connected components of $(X,D)$ and $(Y,D')$ are finite metric spaces.
  \end{example}

  For the following, recall that if $(X,d)$ is an ep-metric space, then $P_{s}(X,d)$ is the poset of subsets $\sigma$ of $X$ such that $d(x,y) \leq s$ for all $x,y \in \sigma$. Recall further that the nerve $BP_{s}(X,d)$ is the barycentric subdivision of the Vietoris-Rips complex $V(X,d)_{s}$, so that the systems $BP(X,d)$ and $V(X,d)$ are naturally sectionwise homotopy equivalent.
  
  \begin{proposition}\label{prop 8}
  Suppose that $i:X \subset Y$ is an inclusion of finite sets. Suppose that $X$ and $Y$ have metric space structures such that $i$ defines a morphism $i: (X,d_{X}) \to (Y,d_{Y})$ of ep-metric spaces. Suppose that for esvery $y \in Y$ there is an $x \in X$ such that $d_{Y}(y,i(x)) < r$ in $Y$.

  Then there are diagrams of poset morphisms
  \begin{equation*}
    \xymatrix{
      P_{s}(X,d_{X}) \ar[r]^-{\sigma} \ar[d]_{i} & P_{m(i)\cdot( s +2r)}(X,d_{X}) \ar[d]^{i} \\
      P_{s}(Y,d_{Y}) \ar[r]_-{\sigma} \ar[ur]^{\theta} & P_{m(i) \cdot (s +2r)}(Y,d_{Y})
    }
    \end{equation*}
for all $0 \leq s < \infty$, in which the upper triangle commutes and the lower triangle homotopy commutes rel $P_{s}(X,d_{X})$.
 \end{proposition}

 \begin{proof}
   Define a function $\theta:Y \to X$ by setting $\theta(x) = x$ for $x \in X$, and by choosing $\theta(y)$ such that $d_{Y}(y,i(\theta(y))) < r$ for $y$ outside of $X$.

   Then 
   \begin{equation*}
     \begin{aligned}
       d_{Y}(i(\theta(y_{1})),i(\theta(y_{2}))) &\leq d_{Y}(i(\theta(y_{1})),y_{1}) + d_{Y}(y_{1},y_{2}) + d_{Y}(y_{2},i(\theta(y_{2}))) \\
       &< d_{Y}(y_{1},y_{2}) +2r,
\end{aligned}
     \end{equation*}
and it follows that
\begin{equation*}
d_{X}(\theta(y_{1}),\theta(y_{2})) \leq m(i)\cdot (d_{Y}(y_{1},y_{2}) +2r).
\end{equation*}

If $\sigma = \{y_{1}, \dots ,y_{n}\}$ is a subset of $Y$ such that $d(y_{j},y_{k}) \leq s$ for all $j,k$, then $\theta(\sigma) = \{\theta(y_{1}), \dots ,\theta(y_{n})\}$ has $d(\theta(y_{j}),\theta(y_{k})) \leq m(i) \cdot (s +2r)$ for all $j,k$.

The subset $\sigma \cup i(\theta(\sigma))$ of $Y$ has distance between any two elements bounded above by $m(i) \cdot (s+2r)$. The natural inclusions
\begin{equation*}
  \sigma \subset \sigma \cup i(\theta(\sigma)) \supset i(\theta(\sigma))
\end{equation*}
  define the required homotopies.
   \end{proof}

     As in Example \ref{ex 7}, suppose that $i: X \subset Y$ is an inclusion of finite sets, and choose systems of neighbourhoods $N_{x}$, $x \in X$ and $N'_{y}$, $y \in Y$, with distances $d_{x}$ and $d'_{y}$. Suppose that $N_{x} \subset N'_{i(x)}$ for all $x \in X$, and that
        \begin{equation*}
          0 \ne d_{i(x)}'(i(x),i(z)) \leq d_{x}(x,z)
\end{equation*}
          for all $z \in N_{x}$, for all $x \in X$. Form the corresponding ep-metric space morphism $i: (X,D) \to (Y,D')$.

          Suppose that $E$ is a global connected component of $(X,D)$ and that $F$ is a global connected component of $(Y,D')$ such that $i(E) \subset F$. Consider the restriction of the ep-metric space morphism $i: (X,D) \subset (Y,D')$ to the ep-metric space morphism $i: (E,D) \to (F,D')$. Suppose that $m(i)$ is the compression factor for the map $i$ of global components.

          The objects $(E,D)$ and $(F,D')$ are metric spaces, by the choices of all weights $d_{x}$ and $d'_{y}$ --- see Example \ref{ex 6}.

The following result is a corollary of Proposition \ref{prop 8}.

\begin{theorem}\label{th 9}
  Suppose that the map $i: (E,D) \to (F,D')$ is the ep-metric space morphism between metric spaces that is described above.
  Suppose that for every $y \in F$ there is an $x \in E$ such that $D'(y,i(x)) < r$. 
 Then there are diagrams
  \begin{equation*}
    \xymatrix{
      P_{s}(E,D) \ar[r]^-{\sigma} \ar[d]_{i} & P_{m(i)\cdot( s +2r)}(E,D) \ar[d]^{i} \\
      P_{s}(F,D') \ar[r]_-{\sigma} \ar[ur]^{\theta} & P_{m(i) \cdot (s +2r)}(F,D')
    }
    \end{equation*}
for all $0 \leq s < \infty$, in which the upper triangle commutes and the lower triangle homotopy commutes rel $P_{s}(E)$.
\end{theorem}

The canonical map
\begin{equation*}
  \phi: V(X,N) = \vee_{x}\ V(X,D_{x}) \to V(X,D),
\end{equation*}
is induced by the ep-metric space maps $(X,D_{x}) \to (X,D)$.

The ep-metric space $(X,D)$ is a disjoint union of its global connected components $E$, and the system $V(X,D)$ is a disjoint union of the systems $V(E,D)$. Form the pullback diagram
\begin{equation*}
  \xymatrix{
    V(X,N)(E) \ar[r] \ar[d]_{\phi} & V(X,N) \ar[d]^{\phi} \\
    V(E,D) \ar[r] & V(X,D)
  }
  \end{equation*}
The disjoint union
\begin{equation*}
  V(X,D) = \bigsqcup_{E \in \pi_{\infty}(X,D)}\ V(E,D)
\end{equation*}
  pulls back to a disjoint union structure
\begin{equation*}
  V(X,N) = \bigsqcup_{E \in \pi_{\infty}(X,D)}\ V(X,N)(E)
\end{equation*}
on $V(X,N)$.

The excision isomorphisms
\begin{equation*}
  \phi_{\ast}: \pi_{0}V(X,N)_{s} \xrightarrow{\cong} \pi_{0}V(X,D)_{s}
\end{equation*}
of Lemma \ref{lem 2} restrict to isomorphisms
\begin{equation}
  \phi_{\ast}: \pi_{0}V(X,N)(E)_{s} \xrightarrow{\cong} \pi_{0}V(E,D)_{s}.
\end{equation}

We finish with a corollary of Theorem \ref{th 9}:

\begin{theorem}\label{th 10}
Suppose that the map $i: (E,D) \to (F,D')$ is the ep-metric space morphism between metric spaces that is described above.
  Suppose that for every $y \in F$ there is an $x \in E$ such that $D'(y,i(x)) < r$. 
 Then there are commutative diagrams
  \begin{equation*}
    \xymatrix{
      \pi_{0}V(X,N)(E)_{s} \ar[r]^-{\sigma} \ar[d]_{i} & \pi_{0}V(X,N)(E)_{m(i)\cdot( s +2r)} \ar[d]^{i} \\
      \pi_{0}V(Y,N')(F)_{s} \ar[r]_-{\sigma} \ar[ur]^{\theta} & \pi_{0}V(Y,N')(F)_{m(i) \cdot (s +2r)}
    }
    \end{equation*}
for all $0 \leq s < \infty$, where $m(i)$ is the compression factor for the map $i$.
  \end{theorem}

Theorem \ref{th 10} is a stability result for clustering in UMAP.

\bibliographystyle{plain}
\bibliography{spt}

\end{document}